\documentclass[review]{elsarticle}
\usepackage{graphics,color}
\usepackage{amsmath}
\usepackage{amsfonts}
\usepackage{amssymb}
\usepackage{graphicx}
\input xy
\xyoption{all}
\usepackage{graphicx,amssymb,amsfonts,latexsym,amsmath,amsthm}

\usepackage{amsfonts}
\usepackage{amssymb}
\usepackage{eucal}
\usepackage{lineno,hyperref}
\modulolinenumbers[5]

\newtheorem{theorem}{Theorem}[section]
\newtheorem{lemma}[theorem]{Lemma}
\newtheorem{proposition}[theorem]{Proposition}
\newtheorem{definition}[theorem]{Definition}

\newcommand{\cF}{{\cal F}}

\newcommand{\cP}{{\cal P}}

\def\cF{{\mathcal {F}}}

\def\cR{{\mathcal R}}

\def\cB{{\mathcal B}}

\def\cP{{\mathcal P}}

\def\di{\diamond}
\def\R{{ \mathbb{R}}}

\everymath{\displaystyle}
\newcommand{\RE} {{\rm I \kern-2.8pt R} }

\begin{document}

\begin{frontmatter}

\title{Evolutionary Dynamics and Lipschitz Maps}

\author{John Cleveland } 
\address{1200 US Highway 14 West, Richland Center WI 53581 }




\begin{abstract}
 In \cite{ CLEVACKTHI, CLEVACK} an attempt is made to find a comprehensive mathematical framework in
which to investigate the problems of well-posedness, asymptotic
analysis and parameter estimation for fully nonlinear evolutionary
game models. A theory is developed as a dynamical system on the state space of finite signed Borel measures under the weak star topology. Two drawbacks of the previous theory is that the techniques and machinery involved in establishing the results are awkward and have not shed light on the parameter estimation question.   For example, in \cite{CLEVACK} the proof for the existence of the dynamical system  is obtained via a fixed point argument using the total variation topology, however, the continuity of the model is established in the $weak^* $ topology. This has caused some confusion. I have remedied this by making all the vital rates Lipschitz and the dynamical system is defined on the dual of the bounded Lipschitz maps, a Banach space. I introduce a method of multiplying a functional by a family of functionals. This multiplication behaves nicely with respect to taking normed estimates. It allows us to form a semiflow that  is locally Lipschitz, positive invariant, and covers all cases: discrete, continuous, pure selection, selection mutation and measure valued models. Under biologically motivated assumptions the model is  uniformly eventually bounded. This remedies both the above problems as only one norm  is used, this norm induces the  $weak^*$ topology on the positive cone of measures and since we have a norm and local Lipschitzity  we can form a theory of Parameter Estimation.
\end{abstract}


\end{frontmatter}


\section{ Introduction}

Evolutionary games (EG)s are a great unifying tool of population dynamics. Models ranging from a basic homogeneous parameter logistic model, to a parametric heterogeneous juvenile adult or consumer resource population model can be modeled effectively \cite{AFT,JC3, JC2}. For these type models, effective modeling means that one can study well posedness, asymptotic analysis and parameter estimation in one abstract setting. An initial attempt was made in \cite{CLEVACK}, however, this involved using several different topologies to establish the dynamical system and it shed no light on the parameter estimation question. We remedy the first problem in this manuscript and provide an introduction to a remedy for the second. The remedy consists of formulating a dynamical system on the dual of the bounded Lipschitz maps and making all of the vital rates Lipschitz continuous mappings.

As a brief recap, we mention again the reasons for the development of this abstract machinery.  We consider the following EG (evolutionary game) model of generalized logistic growth with pure selection (i.e., strategies
replicate themselves exactly and no mutation occurs)  which was developed and
analyzed in \cite{AMFH}:
\begin{equation}
 \frac{d}{dt} x(t,q) = x(t,q) (
q_1 -q_2 X(t)), \label{logiseq}\end{equation}
 where $X(t) = \int_Q x(t,q) dq$ is the total population, $Q \subset \text{int}(\mathbb{R}_+^2)$ is compact
 and the state space is the set of continuous real valued functions
 $C(Q)$. Each $ q=(q_1, q_2) \in Q$ is a two tuple where $q_1$ is an
 intrinsic replication rate and $q_2$ is an intrinsic mortality
 rate. The solution to this model converges to a Dirac
 mass centered at the fittest $q$-class. This is the class with the highest birth to death ratio
 $\frac{q_1}{q_2}$,
 and this convergence is in a topology called $weak^* $
  (point wise convergence of functions) \cite{AMFH}. However, this Dirac limit is not in the
  state space as it is not a continuous function. It is a measure.  Thus, under this formulation one cannot treat this Dirac mass
  as an equilibrium (a constant) solution and hence the study of linear stability analysis is not possible.

  Other examples
  for models developed on classical state spaces such as $L^1(X,\mu)$ that demonstrate the emergence of Dirac measures in the asymptotic limit from smooth initial densities are given in \cite{AFT,AMFH,calsina,CALCAD,GVA,P,GR1,GR2}.
  In particular, how the measures arise naturally in a biological and adaptive dynamics environment is illustrated quite well in \cite[chpt.2]{P}.
   These examples show that the chosen state space for formulating such selection-mutation models must \textbf{contain} densities and Dirac masses and the topology used must \textbf{contain the ability to demonstrate convergence} of densities to Dirac masses.
This process is illustrated in the precursors to this work in \cite{CLEVACKTHI, CLEVACK}. However in this manuscript I shall concentrate on the problems mentioned in the first paragraph.
\begin{definition}
If $X$ is a metric space, and $ J \subset \R_+$ is an interval that contains zero then a map $$ \Phi: J \times X \rightarrow X $$ is called a local (global autonomous) semiflow if:
\begin{itemize}
\item[(1)] $\Phi(0;x) =x.$
\item[(2)] $\Phi(t+s; x) = \Phi(t; \Phi(s;x))$,  $\forall t, s \in J , ~x \in  X.$
\end{itemize}
\end{definition}

 If $f : X \rightarrow X $ is a locally Lipschitz vectorfield and $ x(t)$ is the unique solution to  $x'(t) = f(x)$  and $ x(0) =x_0$. Then we obtain a global autonomous semiflow $\Phi(t; x_0) = x(t). $ This semiflow is always continuous \cite[ Chpt.1, pg.19]{Thi03}.

 In particular, in the present paper we let $ [ X, D_X ] $ be our metric space where  $$ X = BL^* \times L(Q;\mathcal{P}^*). $$ Here $Q$ is a compact metric space and $ BL=BL(Q)$ are the bounded Lipschitz maps on $Q.$ $BL^*$ is the dual of $BL$ and  $ L(Q;\mathcal{P}^*)$ are the Lipschitz maps into $\mathcal{P}^*$. Elements of $\mathcal{P}^*$ are to be thought of as generalizations of probability measures.  They are elements of $BL^*$ of norm 1. $ \gamma \in L(Q;\mathcal{P}^*)$ is the parameter of our system and is to be thought of as a family of ``probability distributions" indexed by $ Q$. It is the \textbf{mutation kernel}. The metric $ D_X$ satisfies
  $$ D_X( (u_1, \gamma_1),( u_2, \gamma_2)) = \|u_1 -u_2 \|_{BL}^* + \|\gamma_1 -\gamma_2\|_{\infty}^*. $$


 (See subsection \ref{technical} for the definitions of  $ \| \cdot \|_{BL}^*$  and $ \|\cdot \|_{\infty}^*.$ )

  In order for a semiflow to model our Evolutionary Game it must satisfy the \textbf{ constraint equations}. In other words our (EG) model is
 an ordered triple $$(Q,\Phi(t;\cdot), \mathcal{F})$$  subject to:

\begin{equation}\label{mconstraint}\frac{d}{dt}\Phi(t;x)[g]= \cF[\Phi(t;x)][g], \text{ for every}
~~g \in BL(Q). \end{equation}
Here $Q$ is the strategy
(compact metric) space, $\Phi(t;x)$
is a semiflow on $X$ and $$\cF : X \rightarrow BL^*$$  is a  vector field (parameter dependent) such that $ \Phi$ and $\cF$
satisfy equation \eqref{mconstraint}. 

Our original models in this field all showed convergence of a semiflow generated from an initial value problem. The equilibrium point was a dirac mass. The obvious choice for state space was $\mathcal{M}_+$, under the $weak^*$ topology. Where $\mathcal{M}_+$  denotes the cone of the positive measures. However, $ \mathcal{M}_+$ is a complete metric space and not a Banach Space.  With slight modifications of the definitions one could use the techniques of either mutational analysis \cite{JPA1,JPA2, Lorenz} or differential equations in metric spaces \cite{Tabor02} or arcflows of arcfields \cite{CC,CB}  to generate a semiflow that satisfies the equivalent of the initial value problem in semiflow theory language.

The method employed here is  that we find a Banach Space,  $BL^*$ containing $ \mathcal{M}_+ $ as a closed metric subspace. Then we extend the constraint equation on $\mathcal{M}_+ $ to one on $BL^*$. The semiflow resulting from the solution of the generalized constraint equation has $ \mathcal{M}_+ $ as a forward invariant subset and hence we generate our semiflow on $ \mathcal{M}_+ $. This is essentially the method employed here. However, using this approach we see that we generate a semiflow on any forward invariant subset of $X$.

The main contributions of the present work
are as follows: \begin{enumerate}
\item  We form a well posed model of a general evolutionary game as a semiflow on a suitable  metric space  that covers discrete, continuous, pure selection, selection mutation, and measure valued models. It should be noted that the pure selection kernel is Lipschitz and Lipschitz continuous function are dense in the continuous ones, since $Q$ is compact.
  \item  Unlike the linear mutation term commonly used in the literature, we allow for nonlinear (density dependent) mutation
term that contains all classical nonlinearities, e.g., Ricker,
Beverton-Holt, Logistic; \item  Unlike the one or two dimensional strategy spaces used in the literature,  we allow for
a strategy space $Q$ that is possibly infinite dimensional. In particular, we assume that $Q$ is a compact metric space. \item  Our state space has a norm and hence all estimates in the state space are performed with one metric which is a norm. This fact allows us to construct a theory of parameter estimation. This latter remedies the problems with the previous approaches.
\end{enumerate}

This paper is organized as follows. In section 2 we demonstrate how
to proceed from a density model to a  $ BL^*$ valued model and thus demonstrate the derivation of the constraint equation.  In section 3
we establish some background material including notation and technical definitions. In section 4 we prove positive invariance and well posedness.
  In section 5 we mention and demonstrate the unifying power of this methodology and mention that with the new formulation we still have both pure selection and selection mutation formulated in a continuous manner since the pure selection kernel is  Lipschitz. In section 6 under a biologically motivated assumption we show uniform eventual boundedness.  In section 7 we provide concluding remarks.


\section{ The Constraint Equation} \label{constraint}
This abbreviated section is taken from \cite{CLEVACK} just for background and the defining of the constraint equation.  For the full account, see \cite{CLEVACK}. We begin with a density version of the constraint equation \eqref{dsmm}.
 To this end take as the strategy space $Q$ a compact subset of ${\rm int}(\mathbb{R}^n_{+})$ (the interior of the positive cone of
$\mathbb{R}^n$)  and  consider the following density IVP (initial value problem) :
\begin{equation} \left\{
\begin{array}{l}
 \frac{d}{dt}x(t,q) =
\underbrace{\int_Q B(X(t),\hat q) p(q,\hat q) x(t,\hat q) d \hat
q}_{\mbox{Birth term}} - \underbrace{D(X(t),q))
x(t,q)}_{\mbox{Mortality term}}\\
 x(0,q)= x_{0}(q).
 \label{dsmm} \end{array}
\right.\end{equation} Here, $X(t) = \int_Q
x(t,q) dq$ is the total population, $B(X,\hat q)$ represents the
density-dependent replication rate per $\hat  q $ individual, while $D(X,q)$ represents the density-dependent mortality rate per $q$ individual.
 The probability density function $p(q,\hat q)$ is the
selection-mutation kernel. That is, $p(q, \hat q) d q$ represents
the probability that an individual of type $\hat q$ replicates an
individual of type $q$ or the proportion of $\hat q$'s offspring
that belong to the $dq$ ball. Hence, $B(X(t),\hat q)p(q, \hat q)
dq$ is the offspring of $\hat q$ in the $dq$ ball and $
B(X(t),\hat q) p(q,\hat q)dq x(t,\hat q) d \hat q$ is the total
replication of the $d\hat q$ ball into the $dq$ ball. Summing
(integrating) over all $ d\hat q$ balls results in the replication
term. Clearly $D(X(t),q) x(t,q)dq$ represents the mortality in the
$dq$ ball. The difference between birth and death in the $dq$ ball
gives the net rate of change of the individuals in the $dq$ ball,
i.e., $\frac{d}{dt}x(t,q)dq .$ Dividing by $dq$ we get \eqref{dsmm}.

We  point out that \textbf{formally}, if we let $p(q,\hat q)
=\delta_{\hat q}(q)=\delta_{q}(\hat q)$ (the delta function is even)
in \eqref{dsmm} then we obtain the following pure selection
(density) model
\begin{equation} \left \{ \begin{array}{l}
\displaystyle \frac{d}{dt}x(t,q) =   x(t,q) (B(X(t),q) -
D(X(t),q))\\
 x(0,q)= x_{0},
 \end{array} \right .
 \label{pureselection}
\end{equation}
of which equation (1) in \cite{AFT} is a special case. Indeed if
$p(q,\hat q)dq = dq\delta_{\hat q} ( q)$ then this means that the
proportion of $\hat q$'s offspring in the $dq$ ball is zero unless
$q= \hat q$ in which case this proportion is $dq,$ i.e., individuals
of type $\hat q$ only give birth to individuals of type $\hat q$.

Multiplying both sides of \eqref{dsmm} by a test function $ g \in C(Q)$ and integrating over $Q$ we obtain:
$$ \int_Qg(q) \frac{d}{dt}x(t,q)dq =  \int_Q
\bigl[\int_Q B(X(t),\hat q) p(q,\hat q) x(t,\hat q) d \hat q  -
D(X(t),q) x(t,q)\bigr]g(q)dq .$$ Changing order of integration we get
$$ \begin{array}{lll}
\int_Qg(q) \frac{d}{dt}x(t,q)dq &=& \int_Q B(X(t),\hat q)\bigl[\int_Q g(q)
p(q,\hat q)dq\bigr]
x(t,\hat q) d \hat q - \int_Q g(q) D(X(t),q) x(t,q)dq\\
&=& \int_Q B(X(t),\hat q)\gamma(\hat q) [g] x(t,\hat q) d \hat q -
\int_Q g(q)D(X(t),q) x(t,q)dq,
\end{array}
$$
where $\gamma(\hat q)[g]=\int_Q  g(q)p(q, \hat q) dq. $ See \cite{CLEVACK} for a more biological interpretation of the mutation kernel.

If $\mu(t)(dq) = x(t,q) dq$  we obtain the following measure valued dynamical system
\begin {equation} \left\{\begin{array}{ll}\label{M1}
 \displaystyle \frac{d}{dt}{\mu}(t, u, \gamma)[g] = \int_Q B(\mu(t)(Q),\hat q ) \gamma(\hat q)[g]\mu(t)(d\hat q)\\
\hspace{1.2 in} - \displaystyle \int_Q  g(\hat q)D( \mu(t)(Q),\hat q)
\mu(t)(d\hat q) =  {F} (\mu, \gamma)[g] \\
\mu(0, u,\gamma)=u.
\end{array}\right.\end{equation}
if $ g \in C(Q).$

If we properly define the $\bullet$ operation below, then we obtain the following  $BL^*$ valued model:
\begin {equation} \left\{\begin{array}{ll}\label{BLvalued}
 \displaystyle \frac{d}{dt}{\mu}(t, u, \gamma)[g] & =  B(\mu(t)(\textbf{1}),\cdot) \gamma(\cdot)\bullet\mu(t) [g]
 - \displaystyle  D( \mu(t)(1),\cdot)\bullet\mu(t)[g] \\
 &= \Bigl(B(\mu(t)(\textbf{1}),\cdot) \gamma(\cdot) [g] -  D( \mu(t)(1),\cdot)\Bigr)\bullet\mu(t)[g]\\
 & =  {F} (\mu, \gamma)[g] \\
\mu(0, u,\gamma)=u.
\end{array}\right.\end{equation}
if $ g \in BL(Q).$

Suppose  $ \mu $ is a solution to \eqref{BLvalued}. Then define
  $$ \Phi(t;(u, \gamma))= [ \mu(t,u,\gamma), \gamma] $$  where $$ \frac{d\Phi}{dt}(t;( u, \gamma))= [\frac{d\mu}{dt}(t,u,\gamma), \gamma]. $$ If
 $ \cF( \mu, \gamma) = [F(\mu, \gamma), \gamma] $ where $F$ is as in \eqref{BLvalued}, then

 \begin{equation*} \begin{split}
 \cF[ \Phi(t;(u, \gamma))]  & = \cF [ \mu(t,u, \gamma), \gamma] =  [F(\mu(t,u, \gamma), \gamma] =[\frac{ d\mu}{dt}(t,u,\gamma), \gamma]
   = \frac{d\Phi}{dt}(t; (u, \gamma))
 \end{split} \end{equation*}
 or for  $ x =( u, \gamma) $
 \begin{equation}\label{CONSTRAINT} \begin{split}    \frac{d\Phi}{dt}(t;x)& = \cF[ \Phi(t; x)] \mbox { ~ and } \\
  \Phi(0;x) = x \end{split} \end{equation}

 \eqref{CONSTRAINT}  is the $BL^*$ valued constraint equation.

\section{Preliminary  Material }\label{PM}

We begin modeling with  $([Q,d], \cB(Q),P)$ where $[Q,d]$ is a compact metric space, $\mathcal{B}
(Q)$ are the Borel sets on $[Q,d]$ and $P$ is a probability measure on the Measurable Space $([Q,d], \mathcal{B}(Q))$ representing an initial weighting on the strategies. One can think of $Q$ as a compact subset of $\mathbb{R}^n$ and $P$ as a probability measure (initial weighting) on this set.
$Q$ above is used to model the space of strategies. What we seek as a model of our game is a semiflow subject to the constraint equation \eqref{CONSTRAINT} which will follow easily from a parameter indexed family of solutions to \eqref{M1} above.

\subsection{Birth and Mortality Rates}
Concerning the birth and mortality densities $B$ and $D$ we make assumptions similar
to those used in \cite{AFT}:
\begin{itemize}
\item[(A1)] $B: \mathbb{R}_+ \times Q \rightarrow \mathbb{R}_+$ is locally Lipschitz
continuous and $B( \cdot,q)$ nonincreasing  on $\mathbb{R}_+$ for any $q \in Q.$
\item[(A2)] $D: \mathbb{R}_+ \times Q \rightarrow \mathbb{R}_+$ is locally Lipschitz continuous,
$ D( \cdot,q)$ is  nondecreasing on $\mathbb{R}_+$ for any $ q \in Q$, and $ \inf_{q \in Q} {D(0, q)} =\varpi
>0 $. (This means that there is some inherent mortality not density
related)
\end{itemize}
These assumptions are of sufficient generality to capture many nonlinearities of classical population dynamics including Ricker,
Beverton-Holt, and Logistic (e.g.,  see \cite{AFT}).

 \subsection{Technical Preliminaries for Bounded Lipschitz Formulation } \label{technical}


If $ [Y, \|\cdot\|_Y] $ is  a Banach Algebra, $ C(Q;Y)$ denotes the continuous $Y$- valued maps under the uniform norm,
 $$  \|f\|_{\infty}   = \sup_{q \in Q} \|f \|_{Y} . $$

Two important subspaces are  $$ L(Q;Y)[M] \subset L(Q;Y) \subset C(Q;Y).$$
Where $L(Q;Y)$ is the dense subspace of all $Y$ -valued   Lipschitz maps and $L(Q;Y)[M]$ is the locally compact subspace of Lipschitz maps with Lipschitz bound smaller than or equal to $M$.

 If no range space is specified then $ C(Q) =C(Q;\mathbb{R})$, denotes the Banach space of continuous real valued functions on $Q$.
 The two important subspaces mentioned above are denoted as $L$ and $L[M]$ respectively.

 $L$ also has a finer structure. Indeed, if $ f \in L, $ define

$$ \|f\|_{Lip}= \sup \left\{ \frac{\|(f(x)-f(y))\|_Y}{d(x,y)}: x,y \in Q, x \neq y \right\}.$$

  Under the  norm $\| f \|_{BL} = \|f\|_{\infty}+ \|f\|_{Lip}$, $L$ becomes a Banach space denoted  $[BL, \|\cdot \|_{BL}]. $

  $ [BL^*, \|\cdot  \|_{BL}^*] $ denotes the continuous dual of this Banach Space and it has a closed convex subspace \begin{equation}\cP ^* = \{ \mu \in BL^*_+ ~| ~ \|\mu\|_{BL}^* = 1 \} . \begin{footnote}{ If $S$ is a subset of a Banach space, then $S_+$ is the intersection of $S$ with the positive cone. } \end{footnote} \end{equation}

  $L$ and $BL$ are the same set, the topology is just different.

Crucial to the success of our modeling efforts is the forming of the parameter space,  $ L(Q;\cP^*) \subset C(Q; BL^*),$ which models the mutation kernel.  It is a convex subset of   $C(Q; BL^*).$ \\

\emph{Some Algebra :} \\

Firstly we note that both $[C(Q;Y), \|\cdot \|_{\infty}] $ and  $ [BL(Q;Y),\|\cdot \|_{BL}] $ are also Banach Algebras and we have the inequality \begin{equation} \label{BA} \|fg \| \le \|f \| \|g \| \end{equation} holding in each space.

Secondly,  we  view $ \gamma \in L(Q; BL^*)$   as a family of bounded linear functionals indexed by $Q$. It has properties that need elucidating for our modeling purposes.  $ L(Q; BL^*)$ is a unital BL- module. Indeed if $f,~ g\in BL,$  $ \gamma \in L(Q;BL^*) $
 \begin{equation} \label{action0} (f \cdot \gamma) (q)[g] =f(q)\gamma(q)[g] \mbox{ and } \| f\gamma \|^*_{\infty}\le \|f\|_{\infty} \|\gamma\|_{\infty}^* \begin{footnote}{ If $ \gamma \in C(Q;BL^*)$, $ \| \gamma\|_{\infty}^* = \sup_{q \in Q} \|\gamma(q)\|_{BL}^* $} \end{footnote}.\end{equation}  We will denote this action simply as $ f\gamma$ since it is just pointwise multiplication. So one can multiply a family of functionals by a Lipschitz map and obtain another family of functionals. Moreover,
the new \textbf{uniform} normed product is no larger than the \textbf{uniformed} product of the norms.

Thirdly,
 $$ L \hookrightarrow L(Q;BL^*)  \mbox { by }   f \mapsto f(\cdot)\delta_{(\cdot)} $$  is an isometry. Where $ \delta_{(\cdot)}$ is the delta functional.

 This allows us to view a Lipschitz function, $f$,  as a \textbf{family} of bounded linear functionals on $ BL$  indexed by $ Q.$ Moreover this viewing preserves the uniform norm, i.e.
  \begin{equation*} \|f\|_{\infty} = \|f(\cdot)\delta_{(\cdot)}\|_{\infty}^*. \end{equation*}

 Fourthly, we need to ``multiply" a functional by a family of functionals.   Let $ M^*_b:= [M_{b}^*(BL;\R), \|\cdot\|^*_{BL}], $ denote the normed  $\mathbb{R}$ -Algebra of bounded maps of $BL$  into $\R$ where we have pointwise addition and multiplication and the norm defined as

  $$ \|\mu \|_{BL}^* = \sup_ {  g \in BL, g \ne 0}   \frac{  | \mu(g)|  }{ \|g\|_{BL}}  $$  

 If $$ \Sigma = [BL (Q; M^*_b), \|\cdot\|_{BL} ]$$  then $\Sigma $ is a $\mathbb{R}$- Algebra  under pointwise addition and multiplication and  $ M^*_b(BL;\R)$ is a $\Sigma$- module. Indeed,  under the action
  $$ \bullet: \Sigma \times M^*_b(BL;\R) \rightarrow M^*_b(BL;\R)$$ given by

 $$ (\gamma\bullet \mu) [g] = \mu \Bigl[ \gamma(\cdot)[g] \Bigr] ~~ \forall g \in BL, ~ \forall \gamma \in \Sigma  $$ we have an action.
 This is a  bounded Lipschitz   functional since $ \forall g \in BL, \gamma(\cdot)[g]$ is bounded and Lipschitz since $ \gamma \in BL(Q;M^*_{b}).$
 With respect to the normed  product we have \begin{equation} \label{action21} \|\gamma \bullet \mu\|_{BL}^* \le \|\gamma\|_{BL} \|\mu \|_{BL}^* .\end{equation}
  Moreover, if $ \mu \in BL^*_{+}$, \eqref{action21} becomes
  \begin{equation} \label{action21b} \|\gamma \bullet \mu\|_{BL}^* \le \|\gamma\|^*_{\infty} \|\mu \|_{BL}^* .\end{equation} where $$ \| \gamma\|^*_{\infty} = \sup_{ q \in Q} \|\gamma(q)\|_{BL}^*.$$

  $ \bullet$ above allows us to ``multiply" a functional,
   $ \mu \in M_b^*$, by a family of functionals $ \gamma \in \Sigma .$

   This new multiplication gives us some important information about our mutation parameter space $L(Q;BL^*).$

    Indeed,

    \begin{itemize}
    \item[(1)] First notice   $$ L(Q;BL^*) \times  BL^* \subset \Sigma \times M_b^*(BL; \R).$$ If we think of  $L(Q;BL^*)$  as $ [ BL(Q;BL^*), \|\cdot\|_{BL}]$ (same set different topology), then we actually have that $$ \bullet:  BL(Q;BL^*)\times BL^* \rightarrow BL^* $$ by
 \begin{equation}\label{action2} (\gamma \bullet \mu)[g]=\mu\Bigl[\gamma(\cdot)[g]\Bigr] .\end{equation}
  The $\bullet$ operation does \textbf{not }make $BL^* $ into a $BL(Q;BL^*)$- module since $BL(Q;BL^*)$ is not a ring .\begin{footnote}{ Since $ \delta_{(\cdot)} $  acts as a sort of identity. It is more than likely that some sort of convolution product could be placed on $BL^*$ with $ \delta_{(\cdot)} $ being the identity. Then it would be a unital  $BL(Q;BL^*)$ module.} \end{footnote}  However, this restriction of $\bullet$ is \textbf{bilinear}.
   \item[(2)] Also note that if  $f \in BL$, then $f\bullet \mu $ is well defined as well. Indeed, from the thirdly observation in the \emph{Some Algebra} section we view $f$ as the family $\gamma(q) =f(q) \delta_q$, and
\begin{equation} \label{identity} (f\bullet \mu)[g] = (f  \delta )\bullet \mu[g] = \mu \Bigl [ f(\cdot)\delta_{( \cdot)}[g]\Bigr ] = \mu  [ f(\cdot) g( \cdot)] = \mu[fg]. \end{equation}
Furthermore \begin{equation} \label{action11} \|f \bullet \mu\|_{BL}^* \le \|f\|_{BL} \|\mu \|_{BL}^* .\end{equation}
\item[(3)] If  $\gamma \in \Sigma / BL^* $, and $\mu \in BL^*$, then $ \gamma \bullet \mu$ is possibly in $ M^*_{b} / BL^* .$ For example, suppose that $ \gamma \in BL^*,$ then $ e^ \gamma( \cdot) \in M^*_b/ BL^*.$ But even though  $\mu \in BL^*$,  $ (e^{\gamma}(\cdot) \bullet \mu)[g] = \mu [ e^ {\gamma(\cdot)[g]}] $ is possibly an element of $\gamma \in M^*_{b} / BL^*.$ For instance, if $ \gamma(\cdot) = \delta_{(\cdot)}$ and  $ \mu = \delta_{q_0},$ for some $ q_0 \in Q.$
\item[(4)] In all cases $ \bullet $ behaves nicely with respect to norm estimation in all norms. The normed product is no larger than the product of the norms.
\end{itemize}

\emph{Miscellaneous:} \\

 If $ \nu \in BL^*$, $$ B_{a}[\nu] = \{ \mu \in BL^* ~|~~ \|\mu -\nu \|_{BL}^* < a \} .$$

 Below $ L\cP^*[M] = L(Q;\mathcal{P}^*)[M]$ and likewise for $BL\cP^*[M] .$

 $\textbf{0}$ denote the zero functional and $1$ denotes the constant function that takes the value $1$.

For any time dependent mapping, $f(t)$,  we let $f' (t)=\frac{df}{dt} (t)$

\section{ Main Well-Posedness Theorem} \label{WP}

The following is the main theorem of this section.
\begin{theorem}\label{main}   Let  $X = {BL^*}  \times L(Q;\mathcal{P}^*).$ Then $[X, D_X] $ is a metric space where
 $$ D_X( (u_1, \gamma_1),( u_2, \gamma_2)) = \|u_1 -u_2 \|_{BL}^* + \|\gamma_1 -\gamma_2\|_{\infty}^*. $$
Moreover there exists a global autonomous semiflow where
$$ \Phi:{\mathbb {R}_+} \times  X  \to X $$  satisfying the following:
\begin{enumerate}
\item  There exists a  continuous mapping $$ \varphi : \R_+ \times  {BL^*}  \times BL(Q;\mathcal{P}^*) \rightarrow BL^* $$ such that
 $$ \Phi(t; (u, \gamma)) = (\varphi(t,u,\gamma), \gamma).$$

\item For fixed $ (u, \gamma ) \in BL^* \times BL(Q;\mathcal{P}^*)$, the mapping $t \mapsto \varphi(t, u,\gamma) \in C( \R_+; BL^*) $ is the unique \emph{solution} to

 \begin {equation} \left\{\begin{array}{ll}\label{MAIN}
 \displaystyle {\mu'}(t) & =
\Bigl (B(\mu(t)(1), \cdot) \gamma(\cdot) - \displaystyle  D(\mu(t)(1),\cdot)\Bigr )\bullet \mu(t)   \\
 & =  {F} (\mu, \gamma) \\
\mu(0)=u.
\end{array}\right.\end{equation} Moreover, if $$ \cF : X \rightarrow X $$ by
$$  \cF[ \mu, \gamma] =[ F(\mu, \gamma), \gamma] \mbox{ ~~ and ~~ } \Phi'(t;(u,\gamma)) = [ \varphi'(t,u, \gamma), \gamma] $$

Then \begin {equation} \left\{\begin{array}{ll}
 \displaystyle {\Phi'}(t;x) & =  \cF [ \Phi(t;x)]    \\
\Phi(0;x)=x.
\end{array}\right.\end{equation}

\item   If   $ X_+ = BL^*_+ \times L(Q; \mathcal{P}^*)$, then $ X_+$  is forward invariant under $\Phi $ i.e. $\Phi(t; X_+) \subset X_+ ,$ $\forall t \in \R_+ $.

\item  $ \forall N \in \mathbb{N},$ if  $ X_N = [0,N] \times B_N[\textbf{0}]_+ \times L\cP^*[N] ,$  then
 $\Phi $ is Lipschitz continuous on $X_N.$

\end{enumerate}
\end{theorem}

We now establish a few results that are needed to prove
Theorem \ref{main}.

 \subsection{Local Existence and Uniqueness of
Dynamical System}\label{WP1}

With this background we prepare to obtain the semiflow that will model our evolutionary game. If $ F$ is the vectorfield defined in \eqref{MAIN} then
  \begin{equation}\label{F} F( \mu, \gamma) = F_1( \mu, \gamma) - F_2(\mu, \gamma) \end{equation} and
  \begin{align} F_1 (\mu, \gamma) & = B(\mu(1), \cdot) \gamma(\cdot)\bullet\mu ~,  &F_2(\mu, \gamma) & = \displaystyle  D(\mu(1),\cdot)\bullet \mu. \end{align}

 For each $ N \in \mathbb{N}$, define $F_{N}$ as follows.
If $j $ is one of the functions $ B,D$  then we extend $j$ to $ \R \times Q $ by setting $j_{N}(x,q) = j(0,q) $ for $x \le 0$
  $ j_{N}(x,q) = j(N,q) $ for $x \ge N
 $. Then $ {j}_{N}: \R \times Q \to \R_+$ is bounded and Lipschitz
continuous.  Let $ {F}_{N} (m, \gamma)
$ be the redefined vector field obtained by replacing $j$ with
$j_{N}.$


For each $(u, \gamma) \in BL^* \times BL(Q;\mathcal{P}^*)$,  we will resolve the following IVP first.

 \begin {equation} \left\{\begin{array}{ll}\label{M2}
  m'(t,u, \gamma) =  F_{N}(m,\gamma)  \\
m(0,u,\gamma)= u.
\end{array}\right.\end{equation}
where  \begin{equation} \label{FN}F_N( \mu, \gamma) = F_{N1}( \mu, \gamma) - F_{N2}(\mu, \gamma) \end{equation} and
  \begin{align} \label{eq:truncation} F_{N1} (\mu, \gamma) & = B_N(\mu(1), \cdot) \gamma(\cdot)\bullet\mu ~,  &F_{N2}(\mu, \gamma) & = \displaystyle  D_N(\mu(1),\cdot)\bullet \mu. \end{align}

\begin{lemma} \label{LF}(Lipschitz $F_N $)

 \begin{itemize}

 \item[(i)] $\forall N \in \mathbb{N}$, there exists continuous $${F}_N : BL^* \times BL(Q; \mathcal{P}^{*}) \rightarrow BL^*.$$
  \item[(ii)] $\forall a>0,$ $ \forall M >0 $, if  $${F}_N:  \overline{ B_{a}[\textbf{0}]_+  } \times L\cP^{*}[M] \rightarrow BL^* $$  or  $${F}_N:  \overline{ B_{a}[\textbf{0}]  } \times BL\cP^{*}[M] \rightarrow BL^* $$  then $ F_N$ is bounded and Lipschitz.
\end{itemize}

\end{lemma}

\begin{proof}

First notice that $(i)$ follows from $(ii)$ since

\begin{equation} BL^* \times  BL(Q; \mathcal{P}^*) = \cup_{N , M \in \mathbb{N}} BL^* \cap \overline{{B_{N}[\textbf{0}]}} \times BL\cP^{*}[M]\begin{footnote}{ See \eqref{eq:union}} \end{footnote} \end{equation}

and   $$  \overline{{B_{N}[\textbf{0}]}} \subset  \overline{{B_{N+1}[\textbf{0}]}} \mbox{~ , ~} BL\cP^*[M] \subset BL\cP^*[M+1] .$$

 We will prove the second condition in $(ii)$. The first is straightforward and the only real difference in the argument used below is that one uses the estimate in \ref{action21} instead of the estimate in \ref{action21b}. If  $ a, M >0$, $ N \in \mathbb{N}$,  $(\zeta, \gamma) $, $(\beta, \lambda) $  $\in \overline{B_{a}[\textbf{0}]}_+\times L\cP^*[M], $  
 then let  $ F_N$ be as in \ref{FN}.  \noindent Then
  \begin{equation}\begin{split}
  F_{N1}(\zeta,\gamma) - F_{N1}(\beta, \lambda) & =
\Bigl [ B_{N}(\zeta(1),\cdot)(\gamma-\lambda)\Bigr ]\bullet
\zeta +  \Bigl (  B_{N}(\zeta(1),\cdot)\\
 & \quad  - B_{N}(\beta(1),\cdot)\Bigr )\lambda(\cdot) \bullet \zeta  +    B_{N}(\beta(1),\cdot)\lambda(\cdot)\bullet (\zeta-\beta)\\
  F_{N2}(\zeta,\gamma) - F_{N2}(\beta, \lambda)  & =   [ D_{N}(\zeta(1),\cdot) -
 D_{N}(\beta(1),\cdot)]\bullet\zeta +  D_{N}(\beta(1),\cdot)\bullet \\
  & \qquad (\zeta-\beta).
  \end{split}
  \end{equation}

Hence,
\begin{equation} \begin{split} \|F_{N1}(\zeta,\gamma)- {F}_{N1}(\beta,\lambda) \|_{BL}^* & \leq \|\gamma  - \lambda \|^*_{\infty} \|B_N(0, \cdot)\|_{\infty} \|\zeta \| _ {BL}^*
 + \|B_N(\cdot,\cdot)\|_{Lip} \| \zeta - \beta \|_{BL}^* \| \zeta \|_{BL}^* \\
 & + \|B_N(\beta(1), \cdot) \|_{BL} \| \lambda\|_{BL}  \| (\zeta-\beta)\|_{BL}^*\\
 \|(F_{N2}(\zeta,\gamma)- {F}_{N2}(\beta,\lambda) )\|_{BL}^* & \leq  \| \zeta - \beta \| _{BL}^* \| \zeta \|_{BL}^* \| D_N (\cdot, \cdot) \|_{BL} +  \| D_N (\cdot, \cdot) \|_{BL}  \| \zeta - \beta \| _{BL}^*.
 \end{split}
\end{equation}

If

 \begin{flalign*}
  B_{\mu}(\zeta, \lambda) & =   \|B_N(\cdot,\cdot)\|_{Lip} \| \| \zeta \|_{BL}^* + \|B_N(\cdot, \cdot) \|_{BL} \| \lambda\|_{BL}  +  \| \zeta \|_{BL}^* \| D_N (\cdot, \cdot) \|_{BL} + \| D_N (\cdot, \cdot) \|_{BL}   \\
 B_{\gamma}(\zeta) & =  \|B_N(0, \cdot)\|_{\infty} \|\zeta \| _ {BL}^*,
\end{flalign*} then
$$ \| F_N (\zeta,\gamma) - F_N (\beta,\lambda)\|_{BL}^* \leq B_{\gamma}(\zeta) \| \gamma - \lambda \|^*_{\infty} + B_{\mu}(\zeta,
 \lambda) \| \zeta -\beta\|_{BL}^* .$$
Since $ (\zeta, \lambda) \in \overline{ B_{a}[\textbf{0}]_+  } \times L\cP^*[M], $ $ F_N$ is bounded and Lipschitz on $ \overline{B_{a}[\textbf{0}]}_+   \times L\cP^*[M].$
 \end{proof}

\begin {lemma}\label{E}(Estimates)  Let  $T>0.$  If $ \zeta, \beta \in C( [0,T];BL^*)  $ and  $ s,t \in  [0,T]$  we have the following estimates:
\begin{enumerate}
%
\item
\begin{itemize}
\item[(a)] \mbox{ As a function of } $q$,
\begin{align}
 \|e^{-\int_{s}^{t}{D}_{N}(\zeta(\tau)(1),q)d\tau}\|_{Lip} & \le \|D_N(\cdot,\cdot)\|_{Lip} T~, & \|e^{-\int_{s}^{t}{D}_{N}(\zeta(\tau)(1),q)d\tau}\|_{\infty} \le 1.
\end{align}
\item[(b)] If  \begin{equation} \begin{split} F(q) &  = e^{-\int_{s}^{t}{D}_{N}(\zeta(\tau)(1),q)d\tau}-e^{-\int_{s}^{t}{D}_{N}(\beta(\tau)(1),q)d\tau} \\
   \|F\|_{\infty}  & \le \|{D}_{N}(\cdot,\cdot)\|_{BL}\int_{s}^{t}\| \zeta(\tau)- \beta(\tau)\|^*_{BL}d\tau .\end{split} \end{equation}
 \end{itemize}
\end{enumerate}
\end{lemma}

\begin{proof}

  \begin{itemize}
\item[(a)]
  Using the mean value theorem on the $C^{\infty}( \mathbb{R})$ function, $e^x$, there exists $\theta(s,t)>0$ such that $$ \begin{array}{lll} && |e^{-\int_{s}^{t} {D}_{N}(\zeta(\tau)(1),\hat q) d\tau}
  - e^{-\int_{s}^{t} {D}_N(\zeta(\tau)(1),q) d\tau} | \\
& \leq &  e^{-\theta} |\int_{s}^{t}
\bigl[{ D}_N(\zeta(\tau)(1), \hat q) -
 {D}_N( \zeta(\tau)(1),q) \bigr] d\tau \bigr | \\
& \leq &    \|{D}_N(\cdot,\cdot)\|_{Lip}T d(\hat q, q).  \\
 \end{array}$$

\item[(b)] Using the mean value theorem on the $C^{\infty}( \mathbb{R})$ function, $e^x$, there exists $ \theta = \theta(s,t) > 0$, such that

 $$ \begin{array}{lll} |F(q)| & = & |e^{-\int_{s}^{t} {D}_{N}(\zeta(\tau)(1),q) d\tau} - e^{-\int_{s}^{t} {D}_N(\beta(\tau)(1),q) d\tau} | \\
& \leq &  e^{-\theta} |\int_{s}^{t}
\bigl[{ D}_N(\zeta(\tau)(1),q) -
 {D}_N(\beta(\tau)(1),q) \bigr] d\tau  | \\
 & \leq &  \int_{s}^{t} \|{D}_N(\cdot,\cdot)\|_{BL}\| \zeta(\tau)- \beta(\tau)\|_{BL}^*d\tau
  .\end{array}$$

\end{itemize}

\end{proof}

\begin{proposition} \label{local}If $  T, M >0 $, $ N \in \mathbb{N}$ let  $ F_N$ be as in \eqref{FN}. There exists a Lipschitz continuous  mapping $$\varphi _{NM} : [0,T] \times B_N[\textbf{0}] \times BL\cP^*[M] \rightarrow BL^*$$  satisfying:

 \begin{enumerate}

  \item For each $ (u, \gamma) \in  B_N[\textbf{0}] \times BL\cP^*[M] $,  $ t \mapsto  \varphi_{NM}(t,u,\gamma) $, is the unique solution to

 \begin {equation} \left\{\begin{array}{ll}\label{eq:M2}
  m'(t) =  F_{N}(m(t),\gamma)  \\
m(0)= u.
\end{array}\right.\end{equation}
in  $ C([0,T]; BL^*).$
\item $ \varphi _{NM} ( [0,T] \times B_N[\textbf{0}]_+ \times L\cP^*[M])\subset BL^*_{+} $
 \item $\varphi_{NM}$  is  Lipschitz continuous on  $ [0,T] \times B_N[\textbf{0}]_+ \times L\cP^*[M].$
\end{enumerate}

  \end{proposition}

\begin{proof}   For $ w \in W = C([0,T] ; BL^*)$ and $\lambda >0 $, define $$ \|w \|_{\lambda} = \sup_{t \in [0,T]} e^{-\lambda t} \|w(t)\|_{BL}^*.$$ It is an exercise to show that $[W, \| \cdot \|_{\lambda}]$ is a Banach space. In fact $ \|\cdot\|_{\infty}$ and
$ \|\cdot\|_{\lambda}$ are equivalent.\\

\emph{Unique local solution to  \eqref{eq:M2}:}\\

 Using  standard techniques for locally Lipschitz vector fields with a parameter into a Banach space, Lemma \ref{LF} relays that we have a unique solution to \eqref{eq:M2} on $ [0,T]$ for any $(u,\gamma) \in BL^* \times  BL(Q; \mathcal{P}^{*}) .$ We can use a Lipschitz argument similar to the one below to show that this mapping is indeed Lipschitz.

 We label this solution $ \varphi_{NM}( \cdot) \equiv \varphi_{NM}(\cdot ,u, \gamma)$ (to denote the dependence on  $(u, \gamma)$ ). \\


\emph{ Forward invariance of $B_N[\textbf{0}]_+ \times L\cP^*[M]$ :} \\

Let $ ( u, \gamma) \in B_{N}[0]_+ \times L\cP^*[M], $ if  $ [W, \|\cdot\|_{\lambda}]$ is as above define

$$ W_{N+} = \{ \zeta \in W ~ | ~ \zeta( [0,T]) \subset ( \overline{B_{\widehat{N}}[0]})_+ \}  \mbox{ where } \widehat{N}> N + \|F_{N1}\|_{\infty}T . $$ Obviously $ W_{N+}$ is a nonempty closed subspace of $W$ and hence is a complete metric space.

If $ \zeta \in W_{N+} $ define

\begin{equation}\label{irep} \begin{array}{l}
(T \zeta)(t) =  e^{-\int _{0}^{t}
D_{N}(\zeta(\tau)(1),\cdot)d\tau} \bullet u
 + \int_{0}^t e^{-\int _{s}^{t}
D_{N}(\zeta(\tau)(1),\cdot)d\tau}\bullet F_{N1}[\zeta(s), \gamma] ds .
\end{array}
\end{equation}

\emph{Contraction Mapping : }\\

From our choice of $ (u, \gamma)$, $$ T :W_{N+} \rightarrow W_{N+}. $$

Indeed, if $ \zeta \in W_{N+}$, then obviously  $T\zeta$  is continuous in $t$.  Furthermore since $F_{N1}[ \cdot, \cdot]$  has the same properties as $ F_N $, namely being uniformly bounded and Lipschitz  on  $B_{\widehat{N}}[\textbf{0}]_+ \times L\cP^*[M]$, we can use \eqref{action21b}, Lemma \ref{LF}  and Lemma \ref{E} to obtain

\begin{equation} \begin{split}
\|(T \zeta)(t)\|_{BL}^*  & \le   \| u\|_{BL}^* +  T \|F_{N1} \|_{\infty}.
\end{split}
\end{equation}
 Hence $T$ is indeed a mapping from  $W_{N+}$ into $W_{N+}.$

Moreover for the above choice of $ (u, \gamma) $, $T$ is a contraction mapping. Indeed, first notice that since
$ u \in BL^*_+, $ if $ g \in BL(Q),~ \|g\|_{BL} \le 1 , $
 \begin{equation} \begin{split}
 ( e^{-\int _{0}^{t}D_{N}( \zeta(\tau)(1), \cdot)d\tau} - e^{-\int _{0}^{t} D_{N}( \beta(\tau)(1), \cdot)d\tau})\bullet u [g]
  & =  u [ ( e^{-\int _{0}^{t}D_{N}( \zeta(\tau)(1), \cdot)d\tau} - e^{-\int _{0}^{t}
D_{N}( \beta(\tau)(1), \cdot)d\tau})g(\cdot)] \\
  \qquad \le  u [ | e^{-\int _{0}^{t}D_{N}( \zeta(\tau)(1), \cdot)d\tau} - e^{-\int _{0}^{t}D_{N}( \beta(\tau)(1), \cdot)d\tau}) &g (\cdot)|]
 \le  \Bigl (\|{D}_{N}(\cdot,\cdot)\|_{BL} \\
\int_{s}^{t}\| \zeta(\tau)- \beta(\tau)\|^*_{BL}d\tau \Bigr )
u[|g(\cdot)|] \le   \Bigl (\|{D}_{N}(\cdot,\cdot)\|_{BL} & \int_{s}^{t}\| \zeta(\tau)- \beta(\tau)\|^*_{BL}d\tau \Bigr ) \|u\|_{BL}^* \\
 \mbox{  The last two estimates use } Lemma ~\ref{E}. \end{split}\end{equation}

Now if $ \zeta, \beta \in W_{N+} , $

\begin{equation} \begin{split}
  T\zeta -T\beta & = ( e^{-\int _{0}^{t}D_{N}(\zeta(\tau)(1), \cdot)d\tau} - e^{-\int _{0}^{t}
D_{N}( \beta(\tau)(1), \cdot)d\tau})\bullet u  \\
  & \quad +  \int_{0}^t e^{-\int _{s}^{t}D_{N}( \zeta(\tau)(1), \cdot)d\tau} \bullet \Bigl ( F_{N1}[\zeta(s), \gamma]- F_{N1}[ \beta(s),\gamma] \Bigr )ds \\
  & \quad + \int_0^t \Bigl( e^{-\int _{s}^{t}
D_{N}( \zeta(\tau)(1), \cdot)d\tau} - e^{-\int _{s}^{t} D_{N}( \beta(\tau)(1), \cdot)d\tau} \Bigr )\bullet F_{N1}[\beta(s), \gamma] ds
  \end{split} \end{equation}
and
  \begin{equation} \begin{split}
  \|T\zeta -T\beta\|_{BL}^* & \le \|D_N(\cdot, \cdot) \|_{BL}\|u\|_{BL}^*  \int_0^t \|\zeta(s) -\beta(s)\|_{BL}^*ds  \\
  & \quad +  \|F_{N1}(\cdot, \cdot)\|_{Lip} (\|D_N(\cdot,\cdot)\|_{BL} T + 1) \int_0^t \|\zeta(s) -\beta(s)\|_{BL}^*ds \\
  & \quad + T \|F_{N1}\|_{\infty}\|D_N(\cdot, \cdot) \|_{BL}  \int_0^t \|\zeta(\tau) -\beta(\tau)\|_{BL}^*ds .
  \end{split} \end{equation}

  If  $ N_T =  \|D_N(\cdot, \cdot) \|_{BL}\|u\|_{BL}^*  + \|F_{N1}(\cdot, \cdot)\|_{Lip}( \|D_N(\cdot,\cdot)\|_{BL} T + 1)+  T \|F_{N1}\|_{\infty}\|D_N(\cdot, \cdot) \|_{BL}, $
\begin{equation}
  e^{-\lambda t} \|(T\zeta)(t) -(T\beta)(t)\|_{BL}^*  \le  N_T \int_0^t e^{-\lambda(t-s)} e^{-\lambda s}\|\zeta(s) -\beta(s)\|_{BL}^*ds.
   \end{equation}
 Hence,

   \begin{equation}\begin{split}
  \|T\zeta -T\beta\|_{\lambda } & \le  N_T \Bigl ( \sup_{t\in [0,T]} \int_0^t e^{-\lambda(t-s)}ds\Bigr ) \|\zeta -\beta\|_{\lambda}  \\
  & \le \frac{N_T}{\lambda} \|\zeta -\beta\|_{\lambda}.
  \end{split} \end{equation}
  Which is a contraction for   $\lambda$  large  enough.

We label this fixed point $ \varphi_{NM+} .$ \\

\emph{Local solution for \eqref{eq:M2} :}\\

 Indeed, using Liebnitz Rule for differentiating under the integral we see that
 \begin{equation} \label{verify} \begin{split}
  \varphi'_{NM+} & = - D_{N}(\varphi(\tau)(1),\cdot) \bullet (e^{-\int _{0}^{t} D_{N}(\varphi_{NM+}(\tau)(1),\cdot)d\tau}  \cdot u)   \\
    &  - D_{N}(\varphi_{NM+}(\tau)(1),\cdot) \bullet \Bigl (\int_{0}^t e^{-\int _{s}^{t}D_{N}(\varphi_{NM+}(\tau)(1),\cdot)d\tau}\bullet F_{N1}[\varphi_{NM+}(s), \gamma] ds \Bigr ) + F_{N1}[ \varphi_{NM+}(t), \gamma] \\
    & = F_{N1}[\varphi_{NM+}(t),\gamma] -D_{N}( \varphi_{NM+}(t)(1), \cdot) \bullet \varphi_{NM+}(t)= F_N[\varphi_{NM+}(t), \gamma]
   \end{split} \end{equation}

Obviously from the integral representation \eqref{irep},  $$\varphi_{NM+} (0;u, \gamma) = u ,  ~\forall u \in B_{N}[\textbf{0}]_+.$$

By uniqueness of solution  $$  \varphi_{NM}(t;u, \gamma) =\varphi_{NM+}(t, u, \gamma) \mbox{ on }  [0,T] \times B_N[\textbf{0}]_+ \times L\cP^*[M] .$$

\emph{ Lipschitz:}\\ 

Looking at the right hand side in \eqref{eq:M2}, we see that $\varphi_{NM}$ is actually $C^{1}([0,T]).$  Moreover, the bound on the derivative only depends on $T$ and $ \|F_N\| _{\infty}. $  Hence  $ \forall (u, \gamma) \in B_N[\textbf{0}]_+ \times L\cP^*[M],$  $ \varphi_{NM}( \cdot,u, \gamma)$ is Lipschitz on $[0,T]$ and  the Lipschitz bound does not depend on the variables $u, \gamma$.

Fix  $ (u_1, \gamma_1), ( u_2, \gamma_2) \in B_{N}[\textbf{0}]_+ \times L\cP^*[M]$, then  $ \varphi_{NM} (\cdot,  u_i, \gamma_i) \in C([0,T];BL^*_+)\mbox{   for  i =1,2.} $

If $ w_i(\cdot) =\varphi_{NM}(\cdot;u_i, \gamma_i)$ for i = 1,2, then

\begin{equation*} \begin{split}
 w_i(t) = u_i + \int_{0}^{t} F_N[ w_i(s), \gamma_i]ds \mbox { for i =1,2 .}
 \end{split} \end{equation*}

Hence
\begin{equation*} \begin{split}
 \|w_1(t) -w_2(t) \|_{BL}^* &  \le \| u_1 -u_2 \|_{BL}^* + \int_{0}^{t}\| F_N[  w_1(s), \gamma_1] - F_N[ w_2(s), \gamma_2] \|_{BL}^*ds \\
 & \le \|u_1-u_2\|_{BL}^* + \|F_N[\cdot, \cdot]  \|_{Lip} \int_{0}^{t} \Bigl (\|w_1(s) -w_2(s)\|_{BL}^* + \| \gamma_1 - \gamma_2\|_{\infty}^* \Bigr) ds
 \end{split} \end{equation*}

and if  $ \lambda > 0 $
 \begin{equation*} \begin{split}
 e^{-\lambda t}\|w_1(t) -w_2(t) \|_{BL}^* &  \le  e^{-\lambda t}\|u_1-u_2\|_{BL}^* + \|F_N[\cdot, \cdot] \|_{Lip} \int_{0}^{t} e^{-\lambda( t-s)} e^{-\lambda s}\|w_1(s) -w_2(s)\|_{BL}^* ds \\
 & \qquad + \|F_N[\cdot, \cdot] \|_{Lip}Te^{-\lambda t} \| \gamma_1 - \gamma_2\|^*_{\infty}.
 \end{split} \end{equation*}

Hence,
 \begin{equation*} \begin{split}
  \|w_1 -w_2 \|_{\lambda} &  \le  \|u_1-u_2\|_{BL}^* + \|F_N[\cdot, \cdot]  \|_{Lip} \sup_{t \in [0,T]} \Bigl( \int_{0}^{t} e^{-\lambda( t-s)} ds\Bigr)  \|w_1 -w_2\|_{\lambda}  \\
 & \qquad + \|F_N[\cdot, \cdot] \|_{Lip}T \| \gamma_1 - \gamma_2\|^*_{\infty}
 \end{split} \end{equation*} and
 \begin{equation*} \begin{split}
 \|w_1 -w_2 \|_{\lambda} &  \le   \frac{ \|F_N[\cdot, \cdot] \|_{Lip} }{ \lambda} \|w_1 -w_2 \|_{\lambda}  +  \|u_1-u_2\|_{BL}^* +\|F_N[\cdot, \cdot] \|_{Lip}  T \| \gamma_1 - \gamma_2\|^*_{\infty}.
 \end{split} \end{equation*}

 If  $\lambda$  is such that  $ \frac{ \|F_N[\cdot, \cdot] \|_{Lip} }{ \lambda} < 1$ then we have

 \begin{equation*} \begin{split}
 \|w_1 -w_2 \|_{\lambda} &  \le   \frac{1}{ (1- \frac{ \|F_N[\cdot, \cdot] \|_{Lip} }{ \lambda})}  (\|u_1-u_2\|_{BL}^* + \|F_N[\cdot, \cdot] \|_{Lip} T \| \gamma_1 - \gamma_2\|^*_{\infty}).
 \end{split} \end{equation*}

 Hence,
 \begin{equation*} \begin{split} \|\varphi(t,u_1, \gamma_) -\varphi(t, u_2, \gamma_2)\|_{BL}^* \le  \frac{e^{\lambda T}}{ (1- \frac{ \|F_N[\cdot, \cdot] \|_{Lip} }{ \lambda})}  (\|u_1-u_2\|_{BL}^* + \|F_N[\cdot, \cdot] \|_{Lip} T \| \gamma_1 - \gamma_2\|^*_{\infty}).
 \end{split} \end{equation*}

 Since $ \varphi_{NM} $ is Lipschitz separately in both $ t $ and  $( u, \gamma) $, it is Lipschitz.\\

\end{proof}

\subsubsection{Proof of Theorem \ref{main}}
\begin{proof} \begin{enumerate} \item If $  T, M >0 $, $ N \in \mathbb{N}$  by Proposition \ref{local}  there exists continuous

 $$\varphi _{NM} : [0,T] \times B_N[\textbf{0}] \times BL\cP^*[M] \rightarrow BL^* .$$

 Since $$ \R_+ \times BL^* \times BL(Q;\mathcal{P}^*) =\bigcup_{ N \in \mathbb{N}} [0,N] \times B_N[0] \times BL\cP^*[N] $$ if we define

 \begin{equation} \label{eq:union} \varphi = \cup \varphi_{NN} \end{equation} then we have our continuous

 $$ \varphi : \R_+ \times  {BL^*}  \times BL(Q;\mathcal{P}^*) \rightarrow BL^* .$$ Furthermore, if $ X = BL^* \times L(Q;\mathcal{P}^*)$ and  $ D_X[(u_1, \gamma_1), ( u_2, \gamma_2)] = \|u_1 - u_2 \|_{BL}^* + \| \gamma_1 - \gamma_2 \|_{\infty}^* ,$
 then  $ [X, D_X]$ is a metric space. Define  $$\Phi: \R_+ \times X \rightarrow X $$  by $$ \Phi(t;(u, \gamma)) =[ \varphi(t,u, \gamma), \gamma] .$$

        \item This also follows from Proposition \ref{local}. Indeed, for fixed $ u, \gamma $ there exists $ \hat{N}$ such that  $ (u,  \gamma) \in B_{\hat{N}}[0] \times BL\cP^*[\hat{N}]$ . Since differentiability is a local condition we only need to verify \eqref{MAIN}  on a finite time interval $[0,N],$   $N \ge \hat N .$ This verification is easily done if we can verify that $\varphi$ is bounded on any such time interval.

                     Indeed suppose that  $ \varphi$ is bounded  on any such time interval. Let
                     $$  N(t) = \| \varphi(t)\|_{BL}^* $$  Then if  $ M > \max \{\sup_{ t \in [0,N]} N(t), N \} $ then on $ [0,N] \times B_N[0] \times BL\cP^*[N] $
                      $$\varphi \equiv \varphi_{NN} \equiv \varphi_{MM}. $$

                      Hence  \begin{equation} \begin{split}
                      \varphi'(t,u, \gamma)  & = \varphi'_{MM}(t,u, \gamma) = F_{M}( \varphi_{MM}(t,u,\gamma), \gamma)   = F_{M} ( \varphi_{NN}(t,u,\gamma), \gamma ) \\
                        & = F ( \varphi_{NN}(t,u,\gamma), \gamma )= F ( \varphi(t,u,\gamma), \gamma).
                      \end{split} \end{equation}

                     Also obviously  $\varphi(0,u, \gamma) = u $

                     $\varphi $ is obviously bounded on any finite interval since it is actually continuous on any finite interval.

                     The argument for the following is found in the section leading up to \eqref{CONSTRAINT}.

 \begin {equation}\label{eq:convenient} \left\{\begin{array}{ll}
 \displaystyle {\Phi'}(t;x) & =  \cF [ \Phi(t;x)]    \\
\Phi(0;x)=x.
\end{array}\right.\end{equation}
So we see that  $ \Phi$  satisfies the constraint equations \eqref{CONSTRAINT}.

\item  \begin{equation*} \begin{split}
 \Phi(\R_+\times X_+)  & =  \Phi(\R_+\times BL^*_{+} \times L(Q;\cP^*))  = \bigcup_{N} \Phi( [0,N]\times B_{N+}[\textbf{0}] \times L\cP^*[N]) \\
& = \bigcup_{N}  \varphi\Bigl([0,N]\times B_{N+}[\textbf{0}] \times L\cP^*[N]\Bigr)\times L\cP^*[N] \subset \bigcup_{N} BL^*_{+}\times L\cP^*[N]\\
&  = X_+
\end{split} \end{equation*}

\item This is an immediate corollary of Proposition \ref{local} given the definition of $ [X,D_X]$ and the fact that $ \varphi$ is locally Lipschitz by Propositon \ref{local}.
  \end{enumerate}
Finally we show that $ \Phi$  is actually a semiflow on $X.$
For the first condition notice that for each $ \gamma \in L(Q;\mathcal{P}^*) $,  $\varphi(\cdot,\cdot, \gamma)$  is a semiflow \cite[ Chpt.1, pg.19]{Thi03}.

Suppose $ x =(u, \gamma) \in X $, then
  \begin{equation} \begin{split}
\Phi(t+s, x) &  =[ \varphi(t+s,u,\gamma), \gamma] = [\varphi(t, \varphi(s,u,\gamma), \gamma), \gamma]= \Phi(t, (\varphi(s,u,\gamma), \gamma)) \\
 & = \Phi(t, \Phi(s,x))\end{split} \end{equation}
The second condition is shown to be satisfied by \eqref{eq:convenient}  above.
\end{proof}

\section{Unification}

Here we demonstrate the unifying power of this method. In \cite{CLEVACK} it is demonstrated how to obtain the  discrete, absolutely continuous, selection mutation and pure selection from a measure theoretic model by a proper choice of initial condition and mutational kernel. Here we demonstrate how to obtain a measure theoretic model and hence we obtain all of the above.

\emph{Measure Valued Constraint Equation:}

  $$ B(\mu(t)(1), \cdot) \gamma(\cdot)\bullet\mu(t)[g] = \mu(t)\Bigl[ B(\mu(t)(1), \cdot) \gamma(\cdot)[g]    \Bigr] = \int_Q  B(\mu(t)(1), q) \gamma(q)[g]\mu(t)(dq)$$

 Hence \begin {equation*} \left\{\begin{array}{ll}
 \displaystyle {\mu'}(t) & = B(\mu(t)(1), \cdot) \gamma(\cdot)\bullet\mu(t) - \displaystyle  D(\mu(t)(1),\cdot)\bullet \mu(t)   \\
 & =  {F} (\mu, \gamma) \\
\mu(0)=u.
\end{array}\right.\end{equation*} becomes

\begin {equation*} \left\{\begin{array}{ll}
 \displaystyle {\mu'}(t) & = \int_Q  B(\mu(t)(1), q) \gamma(q)[\cdot]\mu(t)(dq) - \displaystyle < \int D(\mu(t)(1),q) \mu(t)(dq), \cdot>   \\
 & =  {F} (\mu, \gamma) \\
\mu(0)=u.
\end{array}\right.\end{equation*}

which is the measure valued constraint equation.

\emph{Measure Valued Integral Representation on the Cone :}

 Suppose $u $ is actually in the positive cone on measures, then notice that if   $g \in C(Q)$

\begin{equation} \begin{split}
e^{-\int _{s}^{t} D(\varphi(\tau)(1),\cdot)d\tau}\bullet F_{N1}[\varphi(s), \gamma] [g]
= F_{N1}[\varphi(s), \gamma]\Bigl ( e^{-\int _{s}^{t} D(\varphi(\tau)(1),\cdot)d\tau}g(\cdot) \Bigr )
 =\Bigl [ B( \varphi(s)(1), \cdot) \gamma(\cdot) \bullet \varphi(s) \Bigr ] \\ \Bigl ( e^{-\int _{s}^{t} D(\varphi(\tau)(1),\cdot)d\tau} g(\cdot) \Bigr ) \\
 = \varphi(s) \Bigl [   B( \varphi(s)(1), \cdot) \gamma(\cdot)[e^{-\int _{s}^{t} D(\varphi(\tau)(1),\cdot)d\tau} g(\cdot) ] \Bigr]
 = \int_Q  B( \varphi(s)(1), \hat q) \gamma(\hat q)[ e^{-\int _{s}^{t} D(\varphi(\tau)(1),\cdot)d\tau}g(\cdot) ]\varphi(s)(d\hat q) \\
 = \int_Q  B( \varphi(s)(1), \hat q)  \overline{\Delta}_{s,t, \varphi( \cdot;u, \gamma)} ( \hat q)[ g] \varphi(s)(d\hat q)
 \end{split} \end{equation}

 Hence the integral representation (see \eqref{irep})  becomes

 \begin{equation} \label{irep2}
\varphi(t,u, \gamma)   =   < \int e^{-\int _{0}^{t}D(\varphi(\tau)(Q),q)d\tau} u(dq), \cdot> 
 + \int_{0}^t \int_Q  B( \varphi(s)(Q), \hat q)  \overline{\Delta}_{s,t, \varphi( \cdot;u, \gamma)} ( \hat q)[ \cdot] \varphi(s)(d\hat q)ds
\begin{footnote}{$ {\overline{\Delta}}_{N,s,t, \alpha(\cdot;u,\gamma)}(\hat q)[g]) =\int_{Q}e^{-\int_{s}^{t} D_N(\alpha(\tau,u,\gamma)(1),q)d\tau}g(q)\gamma(\hat q)(dq) $}\end{footnote}
\end{equation}

which is exactly the integral representation for the measure valued semiflow  \cite{CLEVACK, JC3}.

 We mention one more important observation. In \cite{CLEVACK} we notice that the parameter space is  $ C(Q, \mathcal{P}_w)$, but now the parameter space is  $L\cP^*$. In order to model both pure selection and selection mutation in a continuous manner we need for the kernel  $ q \mapsto \delta_q$ to be in $L\cP^*[M]$ for some $M$. This is indeed the case as  \cite[Lemma 3.5]{HILLE},  demonstrates.

\section{Uniform Eventual Boundedness} \label{UEB}
A system $ \frac{dx}{dt} =F(x)$ is called dissipative and its solution uniformly eventually bounded, if all solutions exist for all forward times and if there exists some $c>0$ such that $$ \limsup_{t \rightarrow \infty}||x(t)||< c $$ for all solutions $x$ \cite[pg. 153]{Thi03}.

The following definitions and the assumption are taken from  the manuscript \cite{CLEVACKTHI}.
The {\em reproduction number} of strategy $q\in Q$ at population size $s$ is defined
by
\begin{equation}
\label{eq:rep-num}
\cR( s, q) = \frac{B(s,q)}{D(s,q)}.
\end{equation}
The {\em basic reproduction number} of strategy $q$ is defined by
\begin{equation}
\label{eq:rep-num-basic}
\cR_0(q) = \cR(0,q), \qquad q \in Q.
\end{equation}

The following additional assumption is made.

{ (A3)} For each $q \in Q$ with $\cR_0(q) \ge 1$, there exists a unique
$K(q) \ge 0$ such that $\cR(K(q), q)$ $ =1$.

Since (A1)-(A3) imply that the function $K(\cdot)$ is continuous, it has a maximum and
a minimum on the compact set $Q$. We define
\begin{equation}
\label{Max}
K^\diamond = \max_{ q \in Q} K(q)
\end{equation}
and  \begin{equation} \label{Min} k_{\diamond}= \min_{ q \in Q} K(q).
\end{equation}

\begin{proposition} Assume (A1)- (A3).  $  \forall x \in BL^*_+ \times L(Q; \mathcal{P}^*),$ $ \Phi(\cdot, x)$ is uniformly eventually bounded.

\end{proposition}
\begin{proof} Let $ x =(u, \gamma) \in BL^*_+ \times L(Q; \mathcal{P}^*).$ Since $ u\in BL^*_+$,  $ ~ \varphi(t,u, \gamma) ~= ~ \varphi(t) ~ \in BL^*_+,  \forall t\ge 0.$
From Theorem \ref{main} we have

 \begin {equation*} \left\{\begin{array}{ll}\label{M}
 \displaystyle {\varphi'}(t) & = B(\varphi(t)(1), \cdot) \gamma(\cdot)\bullet \varphi(t) - \displaystyle  D(\varphi(t)(1),\cdot)\bullet \varphi(t)   \\
 & =  {F} (\varphi, \gamma) \\
\varphi(0,u,\gamma)=u.
\end{array}\right.\end{equation*}

Hence \begin{equation}\label{estimate} \begin{split}
 \varphi'(t)(1) & = \varphi(t)\Bigl [ B(\varphi(t)(1), \cdot) \gamma( \cdot)[1]   \Bigr ] -   \varphi(t)\Bigl[ D(\varphi(t)(1),\cdot)     \Bigr ] \\
& = \varphi(t) \Bigl [ B(\varphi(t)(1), \cdot) -  D(\varphi(t)(1),\cdot)    \Bigr ] \\
& = \varphi(t) \Bigl [ (\cR( \varphi(t)(1), \cdot) -1 )D(\varphi(t)(1),\cdot) \Bigr ].
\end{split} \end{equation}

Since  $ \varphi(t) >0 $, $ \varphi(t)(1) =  \| \varphi(t) \|_{BL}^* .$ Using \eqref{estimate}, if $ \varphi(t)(1) > K^{\di}$, then ~  $\varphi'(t)(1) \le 0 .$

Indeed $ \forall q \in Q$ ,  $ \cR (\cdot,q) $ is nonincreasing. Hence  $$\cR (\varphi(t)(1),q) \le  \cR (K^{\di},q) <1  $$

and $$ \limsup_{t \rightarrow \infty}\|\varphi(t)\|_{BL}^* \le K^{\di}  \mbox{ and }  \limsup_{ t \rightarrow \infty} \|\Phi (t;x)\|_{D_X} \le K^{\di} + 1 .$$

\end{proof}

\section{ \bf Concluding Remarks}  In this theory we model an evolutionary game as a semiflow on the metric space  $ X = BL^* \times L(Q;\cal{P}^*) $ of which $ X_+ = BL^*_+ \times L(Q;\cal{P}^*) $ is forward invariant. This model includes all of the well posedness results found in \cite{CLEVACK}.\begin{footnote} {See the list in the second to last paragraph in section 1 above.} \end{footnote} We note that on any forward invariant subspace we have a well-posed model. This includes both  $ \mathcal{M}_+ \mbox { and }  \overline{\mathcal{M}}_+ .$ We conclude that by considering the Lipschitz maps on a compact metric space and forming their dual a nice unifying theory of evolutionary games can be constructed. This elegant theory involves constructing an action $ \bullet$ that allows us to multiply a linear functional by a family of linear functionals. It is difficult to multiply two linear functionals, but it is easy to multiply a linear functional by a family of linear functionals. Moreover this multiplication behaves nicely with respect to norms, i.e. the normed product is less than or equal to the product of the norms. One should notice the length and number of estimates in this paper as compared to those in \cite{CLEVACK,JC3}.

 There are two matters of discussion that I would like to broach that arose during the construction of this manuscript.  Firstly,
 $ [M_b^*(BL;\R), \|\cdot \|_{BL}^*]$ is an extension of $BL^*$ but the set  $M_b^*(BL;\R)$  can be normed with others besides $\|\cdot \|_{BL}^* .$ In particular,
 $$  \| \mu \|_1 = \sup_{\|g\|_{BL} \le 1} |\mu(g)|, \qquad   \| \mu \|_2 = \sup_{\|g\|_{BL} = 1} |\mu(g)| . $$ These (semi) norms coincide on the subspace  $ BL^*. $ However,they do  not necessarily do so on $M_b^*(BL;\R)$. When attempting to uncover fixed points, the particular metric one uses is of utmost importance. More to the point, notice the form of the vector field $F$ in the main equation \eqref{MAIN}

 \begin {equation*} \left\{\begin{array}{ll}
 \displaystyle {\mu'}(t) & =
\Bigl (B(\mu(t)(1), \cdot) \gamma(\cdot) - \displaystyle  D(\mu(t)(1),\cdot)\Bigr )\bullet \mu(t)   \\
 & =  {F} (\mu, \gamma) \\
\mu(0)=u.
\end{array}\right.\end{equation*}

Suppose that one could find a function  $$ \mu(t) = e^{\int_0^t [B(\mu(\tau)(1), \cdot) \gamma( \cdot) - D(\mu(\tau)(1), \cdot) \delta_{(\cdot)}]d\tau} \bullet \mu(0). $$

Here   $$ F(t,\cdot)[g] = e^{\int_0^t [B(\mu(\tau)(1), \cdot) \gamma( \cdot) - D(\mu(\tau)(1), \cdot) \delta_{(\cdot)}]d\tau}[g] =  e^{\int_0^t [B(\mu(\tau)(1), \cdot) \gamma( \cdot)[g] - D(\mu(\tau)(1), \cdot) g(\cdot)]d\tau}. $$
Then notice that \textbf{formally}
\begin{equation}\begin{split}
F_{t}(t,\cdot)[g] & = e^{\int_0^t [B(\mu(\tau)(1), \cdot) \gamma( \cdot)[g] - D(\mu(\tau)(1), \cdot) g(\cdot)]d\tau}[B(\mu(t)(1), \cdot) \gamma( \cdot)[g] - D(\mu(t)(1), \cdot) g(\cdot)]\\
& = e^{\int_0^t [B(\mu(\tau)(1), \cdot) \gamma( \cdot)[\cdot] - D(\mu(\tau)(1), \cdot) ]d\tau}[g][B(\mu(t)(1), \cdot) \gamma( \cdot)[\cdot] - D(\mu(t)(1), \cdot)][g] \\
& = e^{\int_0^t [B(\mu(\tau)(1), \cdot) \gamma( \cdot)[\cdot] - D(\mu(\tau)(1), \cdot) ]d\tau} \bullet[B(\mu(t)(1), \cdot) \gamma( \cdot)[\cdot] - D(\mu(t)(1), \cdot)][g]
\end{split}\end{equation}

So if $ \mu_0$ is linear then
 $$ \mu'(t) = \mu_0( F(t,q) F_t(t,q)) = (B(\mu(t)(1), \cdot) \gamma(\cdot) - \displaystyle  D(\mu(t)(1),\cdot)\Bigr )\bullet \mu(t) $$
 So if we could extend $\bullet$ into a method of multiplying one family of functionals by another family, so as not to lose some of the linearity when linearity is needed then the solution to \eqref{MAIN} would be a exponential. The form of \eqref{MAIN} suggests such a solution and for the pure selection measure valued model, the solution is an exponential, see \cite{AFT}.
 The particular metric that one places on $M_b^*$ will decide if such a fixed point could be obtained.

  Secondly, perhaps the methodology above is not unique. Suppose that $[E,\cal C_E]$  is a pair where $E$ is isometric to a closed convex subset of a Banach Space under its $weak^*$ topology and $\cal C_E$ are constraint equations. Then when is there an extension to $ [B , \cal C_B]$ where $B$ is a Banach Space such that $E \subseteq B $ and  $\cal C_B $ are extensions of the equations in  $\cal C_E $. We want extensions with the property that the semigroup resulting from the resolution of $\cal C_B$ has $E$ as a forward invariant subset.

As far as future development of the theory there are two main paths to be considered. They are asymptotic analysis and parameter estimation. The development of the asymptotic analysis for the measure valued model is well underway in \cite{CLEVACKTHI}. It is anticipated e.g. in section \ref{UEB} that much of those results will be mirrored here as well.  So the main future focus is on parameter estimation.  \cite{AU} reveals how parameter estimation can be performed on structured population models formed on metric spaces metrized with the weak star topology.
So I hope to use the formalism found in \cite{BK} and the techniques found in \cite{AU} to develop a parameter estimation theory for these  $BL^*$ valued models. Formerly the formalism found in \cite{BK} was untenable due to the fact that the model was formed using the total variation norm, which was different from the norm of continuity of the parameter (mutation kernel). However, now this is no longer an obstacle.

\section*{References}
\bibliography{mylib}

\end{document}